\documentclass[a4paper,10pt]{amsart}
\usepackage{amsmath}
\usepackage{amssymb}
\usepackage{latexsym}
\usepackage{amsbsy}
\usepackage[all]{xy}
\usepackage{amscd}
\usepackage[dvips]{graphicx}

\newtheorem{Prop}{Proposition}[section]
\newtheorem{Thm}[Prop]{Theorem}
\newtheorem{Lemma}[Prop]{Lemma}
\newtheorem{Cor}[Prop]{Corollary}
\newtheorem{Example}[Prop]{Example}
\newtheorem{Remark}[Prop]{Remark}

\newtheorem{Definition}[Prop]{Definition}

      \def\ud{{\underline{d}}}

\def\bbn{{\mathbb N}}  \def\bbz{{\mathbb Z}}  \def\bbq{{\mathbb Q}} \def\bb1{{\mathbb 1}}

  \def\bbc{{\mathbb C}}

\def\ra{\rightarrow}
\def\hom{\mbox{Hom}}

\def\ext{\mbox{Ext}\,}

\def\udim{\mbox{\underline {dim}}}

\def\uq2{U_q(\hat{sl}_2)}

\def\bb{{\bf b}}

\def\nd{{\noindent}}

\def\mc{{\mathcal{C}}}

\def\md{{\mathcal{D}}}

\def\ue{{\underline{e}}}

\begin{document}

\title[The cluster characters for cyclic quivers] {The cluster character for cyclic quivers}

\thanks{Fan Xu was supported by
Alexander von Humboldt Stiftung and was also partially supported by
the Ph.D. Programs Foundation of Ministry of Education of China (No.
200800030058)}

\author{Ming Ding and Fan Xu}
\address{Institute for advanced study\\
Tsinghua University\\
Beijing 100084, P.~R.~China} \email{m-ding04@mails.tsinghua.edu.cn
(M.Ding)}
\address{Department of Mathematical Sciences\\
Tsinghua University\\
Beijing 100084, P.~R.~China} \email{fanxu@mail.tsinghua.edu.cn
(F.Xu)}

\maketitle


\bigskip

\begin{abstract}
We define an analogue of the Caldero-Chapoton map (\cite{CC}) for
the cluster category of finite dimensional nilpotent representations
over a cyclic quiver.  We prove that it is a cluster character (in
the sense of \cite{Palu}) and satisfies some inductive
 formulas for the multiplication between the generalized cluster variables (the images of objects of the cluster category under the map). Moreover, we
construct a $\mathbb{Z}$-basis for the algebras generated by all
generalized cluster variables.
\end{abstract}

\setcounter{tocdepth}{1}
\tableofcontents

\section{Introduction}
Cluster algebras were introduced by S. Fomin and A. Zelevinsky
\cite{FZ} in order to develop a combinatorial approach to study
problems of total positivity in algebraic groups and canonical bases
in quantum groups. The link between cluster algebras and
representation theory of quivers were first revealed in \cite{MRZ}.
In~\cite{BMRRT}, the authors introduced the cluster category of an
acyclic quiver $Q$ (a quiver without oriented cycles) as the
categorification of the corresponding cluster algebras. In order to
show that a cluster category categorifies  the involving cluster
algebra, the Caldero-Chapoton map was defined by P. Caldero and F.
Chapoton in \cite{CC}. Let $\mathcal{C}(Q)$ be the cluster category
associated to an acyclic quiver $Q$ (a quiver without oriented
cycles). The Caldero-Chapoton map of an acyclic quiver $Q$ is a map
$$X_?^Q: \mathrm{obj}(\mc(Q))\ra\bbq(x_1,\cdots,x_n).$$
The map was extensively defined by Y. Palu for a Hom-finite
2-Calabi-Yau triangulated categories with a cluster tilting object
(\cite{Palu}).

As in \cite{Keller}, the cluster category can be defined for any
small hereditary abelian category with finite dimensional Hom- and
Ext-spaces. It is interesting to study the cluster categories
without cluster tilting objects and the involving cluster algebras.
For example, the cluster category of a 1-Calabi-Yau abelian category
contains no cluster tilting objects (even no rigid objects).

In this paper, we will focus on the simplest example of the cluster
category without cluster tilting objects: the cluster category of a
cyclic quiver.  We first define an analogue of the Caldero-Chapoton
map for a cyclic quiver. We prove a multiplication formula analogous
to the cluster multiplication theorem for acyclic cluster algebras
(\cite{CK2005}, \cite{XiaoXu}). As a corollary, the map is a cluster
character in the sense of \cite{Palu}. Let $\widetilde{A}_r $ be the
cyclic quiver with $r$ vertices and $\mc(\widetilde{A}_r)$ be its
cluster category. Let $\mathcal{AH}$ be the subalgebra of
$\bbq(x_1,\cdots, x_r)$ generated by $\{X_M \mid M\in\mathrm{
mod}\bbc \widetilde{A}_r\}$ and $\mathcal{EH}$ be the subalgebra of
$\mathcal{AH}$ generated by $\{X_{M}\mid M\in \mathrm{mod}\bbc
\widetilde{A}_r, \ext^1_{\mc{(\widetilde{A}_n)}}(M,M)=0\
 \}.$ We call the algebra $\mathcal{EH}$ the cluster algebra of $\widetilde{A}_r$.
We show that $\mathcal{AH}$ coincides with $\mathcal{EH}$ and
construct a $\bbz$-basis of $\mathcal{EH}$.

\section{The cluster character for cyclic quivers}
Let $k=\bbc$ be the  field of complex number. Throughout the rest
part of this paper, we fix a cyclic quiver $Q=\widetilde{A}_r,$ i.e,
a quiver with oriented cycles, where $Q_0=\{1,2,\cdots, r\}$:
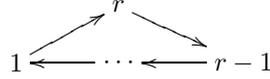
$$
\xymatrix@R=0.8pc{& & r\ar[dr]& &\\
&1\ar[ur]&\cdots\ar[l]& r-1\ar[l]&}
$$
We denote by mod$kQ$ the category of finite-dimensional nilpotent
representations of $kQ$. Let $\tau$ be the Auslander-Reiten
translation functor.  Let $E_1,\cdots,E_r$ be simple modules of the
vertices $1, \cdots, r$, respectively. Set $\udim E_i=s_i$ for $i=1,
\cdots, r.$ We have $\tau E_{2}=E_{1}, \cdots, \tau E_1=E_r$ The
Auslander-Reiten quiver of $kQ$ is a tube of rank $r$ with $E_1,
\cdots, E_r$ lying the mouth of the tube. For $1\leq i\leq r$, we
denote by $E_i=E_{i+r}$, and by $E_i[n]$ the unique nilpotent
indecomposable representation with socle $E_i$ and length $n$. Set
$E_i[0]=0$ for $i=1, \cdots, r.$ We note that that any
indecomposable $kQ$-module is of the form $E_i[j]$ for $i=1, \dots,
r$ and $j\in \bbn\sqcup \{0\}$. Let $\textbf x=\{x_i|i\in Q_0\}$ is
a family indeterminates over $\mathbb{Z}$ and set $x_i=x_{i+mr}$ for
$1\leq i\leq r, m\in \mathbb{Z}_{\geq 0}$. Here we denote by
$\mathbb{Z}_{\geq 0}=\mathbb{Z}\sqcup \{0\}.$

By definition, the cluster category $\mc=\mc(Q)$ is the orbit
category $\md^b(\mathrm{mod} kQ)/\tau\circ [-1]$. It is a
triangulated category by \cite[Theorem 9.9]{Keller}. Different from
the cluster category of an acyclic quiver, the set of objects in
$\mc$ coincides with the set of objects in $\mathrm{mod}kQ$. Also,
for any two indecomposable objects $M, N\in \mc$, we have
$$
\mathrm{Ext}^1_{\mc}(M, N)= \mathrm{Ext}^1_{kQ}(M, N)\oplus
\mathrm{Ext}^1_{kQ}(N, M).
$$
It is possible that both of two terms in the right side don't
vanish. We denote by $\langle -, -\rangle$ the Euler form on
$\mathrm{mod}kQ$, i.e., for any $M, N$ in $\mathrm{mod}kQ$,
$$\langle \udim M, \udim N\rangle:=\mathrm{dim}_{k}\mathrm{Hom}_{kQ}(M, N)-\mathrm{dim}_k\mathrm{Ext}^1_{kQ}(M,
N).
$$ It is well defined. Thus according to the Caldero-Chapoton map, we can similarly
define a map on $\mathrm{mod}kQ$
$$X_?: \mathrm{obj}(\mathrm{mod}kQ)\longrightarrow \mathbb{Z}[\mathbf{x}^{\pm 1}]$$ by
mapping $M$ to
                    $$
                        X_M = \sum_{\underline{e}} \chi(\mathrm{Gr}_{\ue}(M)) \prod_{i \in Q_0} x_i^{-\left<\ue, s_i\right>-\left <s_i, \underline{\mathrm{dim}}M - \ue\right
                        >}
                    $$ where $Gr_{\ue}(M)$ is
the $\ue$-Grassmannian of $M,$ i.e., the variety of
finite-dimensional nilpotent submodules of $M$ with dimension vector
$\ue,$ and set $X_0=1.$ Here, we need not assume that $M$ is
indecomposable by the following Proposition \ref{DM}. The fraction
$X_M$ is called a generalized cluster variable.
\begin{Prop}\label{E[n]}
With the above notation, we have
$$
X_{E_l[n]}=\frac{x_{l+n}}{x_l}+\sum_{k=1}^{n-1}\frac{x_{l+n}x_{l+r-1}}{x_{l+k-1}x_{l+k}}+\frac{x_{l+r-1}}{x_{l+n-1}}.$$
for $n\in \bbn$ and $l=1, \cdots, r.$
\end{Prop}
\begin{proof}
It is known that all submodules of $E_l[n]$ are $E_l[0], E_l[1],
\cdots, E_l[n]$. Set $\underline{d}_{i,j}=\udim E_i[j].$ By
definition,
$$
X_{E_l[n]}=\sum_{k=0}^n\prod_{i\in
Q_0}x_i^{-\left<\underline{d}_{l,k},
\underline{d}_{i,1}\right>-\left <\underline{d}_{i,1},
\underline{d}_{l+k, n-k}\right
                        >}.
$$
By the definition of the Euler form, we have
$$
-\left<\underline{d}_{l,k}, \underline{d}_{i,1}\right>-\left
<\underline{d}_{i,1}, \underline{d}_{l+k,
n-k}\right>=-\mathrm{dim}_{k}\mathrm{Hom}(E_l[k],
E_i)+\mathrm{dim}_{k}\mathrm{Ext}^1(E_l[k], E_i)
$$
$$
-\mathrm{dim}_{k}\mathrm{Hom}(E_i,
E_{l+k}[n-k])+\mathrm{dim}_{k}\mathrm{Ext}^1(E_i, E_{l+k}[n-k]).
$$
If $k=0$, then $\prod_{i\in Q_0}x_i^{-\left<\underline{d}_{l,k},
\underline{d}_{i,1}\right>-\left <\underline{d}_{i,1},
\underline{d}_{l+k, n-k}\right
                        >}=\frac{x_{l+n}}{x_{l}}.$

\nd If $k=n,$ then $\prod_{i\in Q_0}x_i^{-\left<\underline{d}_{l,k},
\underline{d}_{i,1}\right>-\left <\underline{d}_{i,1},
\underline{d}_{l+k, n-k}\right
                        >}=\frac{x_{l+r-1}}{x_{l+n-1}}.$

\nd If $0<k<n,$ then $\prod_{i\in
Q_0}x_i^{-\left<\underline{d}_{l,k},
\underline{d}_{i,1}\right>-\left <\underline{d}_{i,1},
\underline{d}_{l+k, n-k}\right
                        >}=\frac{x_{l+n}x_{l+r-1}}{x_{l+k-1}x_{l+k}}.$
\end{proof}
\begin{Prop}\label{DM}
(1) For $M,N$ in $\mathrm{mod}kQ$, we have
$$
X_{M}X_{N}=X_{M\oplus N}.
$$
(2) Let $0\longrightarrow \tau M\longrightarrow B\longrightarrow
M\longrightarrow 0$ be an almost split sequence in $\mathrm{mod}kQ$,
then
$$
X_{M}X_{\tau M}=X_{B}+1.
$$
\end{Prop}
\begin{proof}
The proof is similar to \cite[Proposition 3.6]{CC}. For (1), by
definition, it is enough to prove that for any dimension vector
$\underline{e}$, we have
$$\chi(Gr_{\underline{e}}(M\oplus N))=\sum_{\underline{f}+\underline{g}=\underline{e}}\chi(Gr_{\underline{f}}(M))\chi(Gr_{\underline{g}}(N)).$$
Consider the natural morphism of varieties
$$
f:
\prod_{\underline{f}+\underline{g}=\underline{e}}Gr_{\underline{f}}(M)\times
Gr_{\underline{g}}(N)\rightarrow Gr_{\underline{e}}(M\oplus N)
$$
defined by sending $(M_1, N_1)$ to $M_1\oplus N_1$. Since $f$ is
monomorphism, we have
$$
\sum_{\underline{f}+\underline{g}=\underline{e}}\chi(Gr_{\underline{f}}(M))\chi(Gr_{\underline{g}}(N))=\chi(\mathrm{Im}f).
$$
On the other hand, we define an action of $\bbc^*$ on
$Gr_{\underline{e}}(M\oplus N)$ by  $$t. (m, n)=(tm, t^2n)$$ for
$t\in \bbc^*$ and $m\in M, n\in N$. The set of stable points is just
$\mathrm{Im}f.$ Hence,
$\chi(\mathrm{Im}f)=\chi(Gr_{\underline{e}}(M\oplus N)).$
This proves (1).\\
(2)\ Assume that $M=E_i[n]$. It is enough to prove:
$$X_{E_{1}[n]}X_{E_{2}[n]}=X_{E_{1}[n+1]}X_{E_{2}[n-1]}+1$$where $0\longrightarrow E_{1}[n]\longrightarrow E_{2}[n-1]\oplus E_{1}[n+1]\longrightarrow
E_{2}[n]\longrightarrow 0$ is an almost split sequence in
$\mathrm{mod}kQ$. The equation in (2) follows from the direct
confirmation by using Proposition \ref{E[n]}.
\end{proof}

Let $M, N$ be indecomposable  $kQ-$modules satisfying that
$\mathrm{dim}_{k}\mathrm{Ext}^1_{kQ}(M,
N)=\mathrm{dim}_{k}\mathrm{Hom}_{kQ}(N, \tau M)=1$. Assume that
$M=E_i[j]$, $N=E_k[l]$. Then in $\mc(Q)$, there are just two
involving triangles
$$
E_{k}[l]\rightarrow E\rightarrow E_{i}[j]\rightarrow \tau E_{k}[l]
$$
and
$$
E_{i}[j]\rightarrow E'\rightarrow E_{k}[l]\xrightarrow{g} \tau
E_{i}[j]
$$
where $E\cong E_{k}[i+j-k]\oplus E_{i}[k+l-i]$ and $E'\cong
\mathrm{ker}g\oplus \tau^{-1}\mathrm{coker}g.$
\begin{Thm}\label{exp}
 With the above notation, we have
$$
X_MX_N=X_{E}+X_{E'}.
$$
\end{Thm}

\begin{proof}
Assume that $$X_M = \sum_{\underline{e}} \chi(\mathrm{Gr}_{\ue}(M))
\prod_{i \in Q_0} x_i^{-\left<\ue, s_i\right>-\left <s_i,
\underline{\mathrm{dim}}M - \ue\right
                        >}$$ and
                        $$
X_N = \sum_{\underline{e}'} \chi(\mathrm{Gr}_{\ue'}(N)) \prod_{i \in
Q_0} x_i^{-\left<\ue', s_i\right>-\left <s_i,
\underline{\mathrm{dim}}N - \ue'\right
                        >}
                        $$
Then
$$
X_MX_N=\sum_{\ue, \ue'}\prod_{i\in Q_0}x_i^{-\left<\ue+\ue',
s_i\right>-\left <s_i, \underline{\mathrm{dim}}M
+\underline{\mathrm{dim}N}- (\ue+\ue')\right>}.
$$
Note that $\chi(\mathrm{Gr}_{\ue}(L))=1$ or $0$ for any $kQ$-module
$L.$ Since $\mathrm{Ext}^1_{kQ}(M, N)\neq 0$, we have a short exact
sequence
$$
0\rightarrow E_{k}[l]\xrightarrow{f_1} E\xrightarrow{f_2}
E_{i}[j]\rightarrow 0.
$$
Define a morphism of varieties
$$
\phi: Gr_{\ud}(E)\rightarrow
\bigsqcup_{\ue+\ue'=\ud}Gr_{\ue}(M)\times Gr_{\ue'}(N)
$$
by sending $(E_1)$ to $(f_1^{-1}(E_1), f_2(E_1)).$ For $(M_1,
N_1)\in Gr_{\ue}(M)\times Gr_{\ue'}(N)$, we consider the natural
map:
$$ \beta': \ext^{1}_{kQ}(M,N)\oplus \ext^{1}_{kQ}(M_1,N_1)\rightarrow
\ext^{1}_{kQ}(M_1,N) $$  sending $(\varepsilon,\varepsilon')$ to
$\varepsilon_{M_1}-\varepsilon'_{N}$ where $\varepsilon_{M_1}$ and
$\varepsilon'_{N}$ are induced by including $M_1\subseteq M$ and
$N_1\subseteq N,$ respectively and the projection
$$
p_0: \ext^{1}_{kQ}(M,N)\oplus \ext_{kQ}^{1}(M_1,N_1)\rightarrow
\ext^{1}_{kQ}(M,N).
$$
It is easy to check that $(M_1, N_1)\in \mathrm{Im}\phi$ if and only
if $p_0(\mathrm{ker}\beta')\neq 0.$ Hence, we have
$$
X_E=\sum_{\ue, \ue'; p_0(\mathrm{ker}\beta')\neq 0}\prod_{i\in
Q_0}x_i^{-\left<\ue+\ue', s_i\right>-\left <s_i,
\underline{\mathrm{dim}}M +\underline{\mathrm{dim}N} -
(\ue+\ue')\right>}.
$$

Assume that $$X_{E'} = \sum_{\underline{d}'_1, \ud'_2}
\chi(\mathrm{Gr}_{\ud'_1}(K))
\chi(\mathrm{Gr}_{\ud'_2}(\tau^{-1}C))\prod_{i \in Q_0}
x_i^{-\left<\ud'_1+\ud'_2, s_i\right>-\left <s_i,
\underline{\mathrm{dim}}K+ \underline{\mathrm{dim}}\tau^{-1}C-
\ud'_1-\ud'_2\right
                        >} $$
Set
$\ud^*=\underline{\mathrm{dim}}M-\underline{\mathrm{dim}}\tau^{-1}C.$
We have
\begin{eqnarray}
    && \left <s_i,
\underline{\mathrm{dim}}K+ \underline{\mathrm{dim}}\tau^{-1}C\right
                        >  \nonumber\\
   &=& \left <s_i,
\underline{\mathrm{dim}}N-\tau(\ud^*)+
\underline{\mathrm{dim}}M-\ud^*\right
                        >\nonumber\\
  &=& \left <s_i,
\underline{\mathrm{dim}}M+\underline{\mathrm{dim}}N-\ud^*\right
                        >+\left <\ud^*, s_i \right >. \nonumber
\end{eqnarray}
Hence, $X_{E'}$ can be reformulated as
$$
\sum_{\underline{d}'_1, \ud'_2} \prod_{i \in Q_0}
x_i^{-\left<\ud'_1+\ud'_2+\ud^*, s_i\right>-\left <s_i,
\underline{\mathrm{dim}}M+ \underline{\mathrm{dim}}N-
(\ud'_1+\ud'_2+\ud^*)\right
                        >}
$$

Since $\mathrm{dim}_{k}\mathrm{Hom}_{kQ}(N, \tau M)=1$, there is
only one element in $\mathbb{P}\mathrm{Hom}_{kQ}(N, \tau M)$ with
the representative $g$. We have a long exact sequence
$$
\xymatrix{0\ar[r]&K\ar[r]&N\ar[r]^g&\tau M\ar[r]&C\ar[r]&0}
$$
 Given submodules $K_1, C_1$ of $K, C$, respectively,
we have the commutative diagram
$$
\xymatrix{ 0\ar[r]& K\ar[r]\ar[d]& N\ar[r]^{g}\ar[d]& \tau M\ar[r]&
C\ar[r]& 0\\
0\ar[r]& K/K_1\ar[r]& N/K_1\ar[r]^-{g'}& \tau M_1\ar[r]\ar[u] &
C_1\ar[r]\ar[u]& 0}
$$
where $\tau M_1$ is the corresponding pullback. Define a morphism of
varieties
$$\phi': \bigsqcup_{ \ud_1'+\ud_2'+\ud^*=\ud'}Gr_{\ud'_1}(K)\times Gr_{\ud'_2}(\tau^{-1}C)\rightarrow
\bigsqcup_{\ue+\ue'=\ud'}Gr_{\ue}(M)\times Gr_{\ue'}(N) $$ by
sending $(K_1, \tau^{-1}(C_1))$ to $(K_1, M_1)$. Checking the above
diagram, we know that $(M_1, N_1)\in \mathrm{Im}\phi'$ if and only
if $\mathrm{Hom}_{kQ}(N/N_1, \tau M_1)\neq 0.$ Therefore, we obtain
$$
X_{E'}=\sum_{\ue, \ue'; \mathrm{Hom}_{kQ}(N/N_1, \tau M_1)\neq
0}\prod_{i\in Q_0}x_i^{-\left<\ue+\ue', s_i\right>-\left <s_i,
\underline{\mathrm{dim}}M +\underline{\mathrm{dim}N} -
(\ue+\ue')\right>}.
$$

Consider the dual of $\beta'$:
$$
\beta: \mathrm{Hom}_{kQ}(N, \tau M_1)\rightarrow
\mathrm{Hom}_{kQ}(N, \tau M)\oplus \mathrm{Hom}_{kQ}(N_1, \tau M_1).
$$
Then $$ (p_0(\mathrm{ker}\beta'))^{\perp}= \mathrm{Im}\beta\bigcap
\hom(N,\tau M)\simeq \hom(N/N_1,\tau M_1).
$$ We obtain that
$$
\mathrm{dim}_{k}(p_0(\mathrm{ker}\beta'))+\mathrm{dim}_{k}\mathrm{Hom}(N/N_1,
\tau M_1)=\mathrm{dim}_{k}\mathrm{Ext}^1_{kQ}(M, N)=1
$$
Hence, any $(M_1, N_1)$ belongs to either $\mathrm{Im}\phi$ or
$\mathrm{Im}\phi'$ for some $\ud$ or $\ud'$. We complete the proof.
\end{proof}

Following the definition of a cluster character in \cite{Palu}, we
can easily  check the following corollary.
\begin{Cor}
The Caldero-Chapoton map for a cyclic quiver is a cluster character.
\end{Cor}
We will construct some inductive formulas in the next section. For
convenience,  we write down the following corollary.
\begin{Cor}\label{DM1}
With the above notation, we have
$$
(1)\ X_{E_{i+n}}X_{E_{i}[n]}=X_{E_{i}[n+1]}+X_{E_{i}[n-1]}
$$
$$
(2)\ X_{E_i}X_{E_{i+1}[n]}=X_{E_{i}[n+1]}+X_{E_{i+2}[n-1]}.
$$
\end{Cor}

\section{Inductive multiplication formulas}
In this section, we will give inductive multiplication formulas for
any two generalized cluster variables on $\mathrm{mod}kQ$. Note that
these inductive multiplication formulas are an analogue of those for
tubes in \cite {DXX} for acyclic cluster algebras.

\begin{Thm}\label{16}
Let $i,j, k,l,m$ and $r$ be in $\bbz$ such that $1\leq k\leq
mr+l,0\leq l\leq r-1,1\leq i,j\leq r,m\geq 0$.

\nd (1)When $j\leq i$, then

1)for $k+i\geq r+j$, we have
$X_{E_i[k]}X_{E_j[mr+l]}=X_{E_i[(m+1)r+l+j-i]}X_{E_j[k+i-r-j]}+X_{E_i[r+j-i-1]}X_{E_{k+i+1}[(m+1)r+l+j-k-i-1]},$

2)for $k+i< r+j$ and $i\leq l+j\leq k+i-1$, we have
$X_{E_i[k]}X_{E_j[mr+l]}=X_{E_j[mr+k+i-j]}X_{E_i[l+j-i]}+X_{E_j[mr+i-j-1]}X_{E_{l+j+1}[k+i-l-j-1]},$

3)for other conditions, i.e, there are no extension between $E_i[k]$
and $E_j[mr+l]$, we have $X_{E_i[k]}X_{E_j[mr+l]}=X_{E_i[k]\oplus
E_j[mr+l]}$.

\nd (2)When $j> i$, then

1)for $k\geq j-i,$ we have
$X_{E_i[k]}X_{E_j[mr+l]}=X_{E_i[j-i-1]}X_{E_{k+i+1}[mr+l+j-k-i-1]}+X_{E_i[mr+l+j-i]}X_{E_j[k+i-j]},$

2)for $k< j-i$ and $i\leq l+j-r\leq k+i-1$, we have
$X_{E_i[k]}X_{E_j[mr+l]}=X_{E_j[(m+1)r+k+i-j]}X_{E_i[l+j-r-i]}+X_{E_j[(m+1)r+i-j-1]}X_{E_{l+j+1}[k+r+i-l-j-1]},$

3)for other conditions, i.e, there are no extension between $E_i[k]$
and $E_j[mr+l]$, we have $X_{E_i[k]}X_{E_j[mr+l]}=X_{E_i[k]\oplus
E_j[mr+l]}.$
\end{Thm}

\begin{proof}
We only prove (1) and (2) is totally similar to (1).

1) When $k=1,$ by $k+i\geq r+j$ and $1\leq j\leq i\leq
r\Longrightarrow i=r$ and $j=1.$\\
Then by Proposition \ref{DM} and Corollary \ref{DM1}, we have
$$X_{E_r}X_{E_1[mr+l]}=X_{E_r[mr+l+1]}+X_{E_2[mr+l-1]}.$$
 When $k=2,$ by $k+i\geq r+j$ and $1\leq j\leq i\leq
r\Longrightarrow
i=r\ or\ i=r-1.$\\
For $i=r\Longrightarrow j=1\ or\ j=2$:\\
The case for $i=r$ and $j=1$, we have
\begin{eqnarray}
    && X_{E_r[2]}X_{E_1[mr+l]}  \nonumber\\
   &=& (X_{E_r}X_{E_1}-1)X_{E_1[mr+l]} \nonumber\\
  &=& X_{E_1}(X_{E_r[mr+l+1]}+X_{E_2[mr+l-1]})-X_{E_1[mr+l]} \nonumber\\
  &=& X_{E_1}X_{E_r[mr+l+1]}+(X_{E_1[mr+l]}+X_{E_3[mr+l-2]})-X_{E_1[mr+l]} \nonumber\\
  &=& X_{E_1}X_{E_r[mr+l+1]}+X_{E_3[mr+l-2]}.\nonumber
\end{eqnarray}
The case for $i=r$ and $j=2$, we have
\begin{eqnarray}
&& X_{E_r[2]}X_{E_2[mr+l]} \nonumber\\
  &=& (X_{E_r}X_{E_1}-1)X_{E_2[mr+l]} \nonumber\\
  &=& X_{E_r}(X_{E_1[mr+l+1]}+X_{E_3[mr+l-1]})-X_{E_2[mr+l]} \nonumber\\
  &=& X_{E_r[mr+l+2]}+(X_{E_2[mr+l]}+X_{E_r}X_{E_3[mr+l-1]})-X_{E_2[mr+l]} \nonumber\\
  &=& X_{E_r[mr+l+2]}+X_{E_r}X_{E_3[mr+l-1]}.\nonumber
\end{eqnarray}

For $i=r-1\Longrightarrow j=1$:
\begin{eqnarray}
 && X_{E_{r-1}[2]}X_{E_1[mr+l]} \nonumber\\
  &=& (X_{E_{r-1}}X_{E_r}-1)X_{E_1[mr+l]} \nonumber\\
  &=& X_{E_{r-1}}(X_{E_r[mr+l+1]}+X_{E_2[mr+l-1]})-X_{E_1[mr+l]} \nonumber\\
  &=& (X_{E_{r-1}[mr+l+2]}+X_{E_1[mr+l]})+X_{E_{r-1}}X_{E_2[mr+l-1]}-X_{E_1[mr+l]} \nonumber\\
  &=& X_{E_{r-1}[mr+l+2]}+X_{E_{r-1}}X_{E_2[mr+l-1]}.\nonumber
\end{eqnarray}\\
Now, suppose it holds for $k\leq n,$ then by induction we have
\begin{eqnarray*}
   && X_{E_i[n+1]}X_{E_j[mr+l]}\nonumber \\
   &=& (X_{E_i[n]}X_{E_{i+n}}-X_{E_i[n-1]})X_{E_j[mr+l]}\nonumber \\
   &=& X_{E_{i+n}}(X_{E_i[n]}X_{E_j[mr+l]})-X_{E_i[n-1]}X_{E_j[mr+l]} \nonumber\\
   &=& X_{E_{i+n}}(X_{E_i[(m+1)r+l+j-i]}X_{E_j[n+i-r-j]}+X_{E_i[r+j-i-1]}X_{E_{n+i+1}[(m+1)r+l+j-n-i-1]})\nonumber \\
   && -(X_{E_i[(m+1)r+l+j-i]}X_{E_j[n+i-r-j-1]}+X_{E_i[r+j-i-1]}X_{E_{n+i}[(m+1)r+l+j-n-i]}) \nonumber\\
   &=& X_{E_i[(m+1)r+l+j-i]}(X_{E_j[n+i+1-r-j]}+X_{E_j[n+i-r-j-1]}) \nonumber\\
   && +X_{E_i[r+j-i-1]}(X_{E_{n+i}[(m+1)r+l+j-n-i]}
+X_{E_{n+i+2}[(m+1)r+l+j-n-i-2]}) \nonumber\\
   && -(X_{E_i[(m+1)r+l+j-i]}X_{E_j[n+i-r-j-1]}+X_{E_i[r+j-i-1]}X_{E_{n+i}[(m+1)r+l+j-n-i]})\nonumber \\
   &=& X_{E_i[(m+1)r+l+j-i]}X_{E_j[n+i+1-r-j]}+X_{E_i[r+j-i-1]}X_{E_{n+i+2}[(m+1)r+l+j-n-i-2]}.
\end{eqnarray*}

2)  When $k=1,$ by $i\leq l+j\leq k+i-1\Longrightarrow i\leq
l+j\leq i\Longrightarrow i=l+j.$\\
Then by  by Proposition \ref{DM} and Corollary \ref{DM1}, we have
$$X_{E_{i}}X_{E_j[mr+l]}=X_{E_{l+j}}X_{E_j[mr+l]}=X_{E_j[mr+l+1]}+X_{E_j[mr+l-1]}$$
  When $k=2,$ by $i\leq l+j\leq k+i-1\Longrightarrow i\leq
l+j\leq i+1\Longrightarrow i=l+j\ or\ i+1=l+j$:\\
For $i=l+j$, we have
\begin{eqnarray*}
   && X_{E_{i}[2]}X_{E_j[mr+l]} \nonumber\\
   &=& X_{E_{l+j}[2]}X_{E_j[mr+l]}\nonumber \\
   &=& (X_{E_{l+j}}X_{E_{l+j+1}}-1)X_{E_j[mr+l]}\nonumber \\
   &=& (X_{E_j[mr+l+1]}+X_{E_j[mr+l-1]})X_{E_{l+j+1}}-X_{E_j[mr+l]}\nonumber \\
   &=& X_{E_j[mr+l+2]}+X_{E_j[mr+l]}+X_{E_{l+j+1}}X_{E_j[mr+l-1]}-X_{E_j[mr+l]}\nonumber \\
   &=& X_{E_j[mr+j+2]}+X_{E_{l+j+1}}X_{E_j[mr+l-1]}.
\end{eqnarray*}

For $i+1=l+j$, we have
\begin{eqnarray*}
   && X_{E_{i}[2]}X_{E_j[mr+l]}\nonumber \\
   &=& X_{E_{l+j-1}[2]}X_{E_j[mr+l]}\nonumber \\
   &=& (X_{E_{l+j-1}}X_{E_{l+j}}-1)X_{E_j[mr+l]}\nonumber \\
   &=& (X_{E_j[mr+l+1]}+X_{E_j[mr+l-1]})X_{E_{l+j-1}}-X_{E_j[mr+l]} \nonumber \\
   &=& X_{E_j[mr+l+1]}X_{E_{l+j-1}}+(X_{E_j[mr+l]}+X_{E_j[mr+l-2]})-X_{E_j[mr+l]}\nonumber \\
   &=& X_{E_j[mr+l+1]}X_{E_{l+j-1}}+X_{E_j[mr+l-2]}.
\end{eqnarray*}

Suppose it holds for $k\leq n,$ then by induction we have
\begin{eqnarray*}
   && X_{E_i[n+1]}X_{E_j[mr+l]}\nonumber \\
   &=& (X_{E_i[n]}X_{E_{i+n}}-X_{E_i[n-1]})X_{E_j[mr+l]}\nonumber \\
   &=& (X_{E_i[n]}X_{E_j[mr+l]})X_{E_{i+n}}-X_{E_i[n-1]}X_{E_j[mr+l]}\nonumber\\
   &=& (X_{E_j[mr+n+i-j]}X_{E_i[l+j-i]}+X_{E_j[mr+i-j-1]}X_{E_{l+j+1}[n+i-l-j-1]})X_{E_{i+n}} \nonumber\\
   && -(X_{E_j[mr+n+i-j-1]}X_{E_i[l+j-i]}+X_{E_j[mr+i-j-1]}X_{E_{l+j+1}[n+i-l-j-2]})\nonumber \\
   &=& (X_{E_j[mr+n+i+1-j]}+X_{E_j[mr+n+i-j-1]})X_{E_i[l+j-i]} \nonumber\\
   && +(X_{E_{l+j+1}[n+i-l-j]}+X_{E_{l+j+1}[n+i-l-j-2]})X_{E_j[mr+i-j-1]} \nonumber\\
   && -(X_{E_j[mr+n+i-j-1]}X_{E_i[l+j-i]}+X_{E_j[mr+i-j-1]}X_{E_{l+j+1}[n+i-l-j-2]})\nonumber \\
   &=& X_{E_j[mr+n+i+1-j]}X_{E_i[l+j-i]}+X_{E_j[mr+i-j-1]}X_{E_{l+j+1}[n+i-l-j]}.
\end{eqnarray*}

3)  It is trivial by the definition of the  Caldero-Chapoton map.
\end{proof}

\section  {A $\mathbb{Z}$-basis for cyclic quivers}

In this section, we will focus on studying  the following set
$$\mathcal{B}(Q)=\{X_{R}|\mathrm{Ext}_{k Q}^{1}(R,R)=0\}.$$
 We prove that $\mathcal{B}(Q)$ is a $\mathbb{Z}$-basis of the
algebra $\mathcal{AH}(Q)$ generated by all these generalized cluster
variables. We first give the following definition.
\begin{Definition}\label{p}
For $M, N\in$ $\mathrm{mod} kQ$ with
$\mathrm{\underline{dim}}M=(m_{1},\cdots,m_{r})$ and
$\mathrm{\underline{dim}}N=(s_{1},\cdots,s_{r})$, we write
$\mathrm{\underline{dim}}M\preceq \mathrm{\underline{dim}}N$ if
$m_{i}\leq s_{i}\ for\ 1\leq i\leq r$. Moreover, if there exists
some i such that $m_{i}< s_{i}$, then we write
$\mathrm{\underline{dim}}M\prec \mathrm{\underline{dim}}N.$
\end{Definition}

\begin{Remark}\label{7}
It is easy to see that $\mathrm{\underline{dim}}E_{i+2}[n-1]\prec
\mathrm{\underline{dim}}E_{i}[n+1]$ and
$\mathrm{\underline{dim}}E_{i}[n-1]\prec
\mathrm{\underline{dim}}E_{i}[n+1]$ in Corollary \ref{DM1}.
\end{Remark}

\begin{Lemma}\label{6}
Let $T_1, T_2$ be $kQ$-modules such that
$\mathrm{\underline{dim}}T_1=\mathrm{\underline{dim}}T_2$. Then we
have
$$X_{T_1}=X_{T_2}+\sum_{\mathrm{\underline{dim}}R\prec
\mathrm{\underline{dim}} T_2}a_{R}X_{R}$$ where $R\in
\mathrm{mod}kQ$ and $a_{R}\in \mathbb{Z}$.
\end{Lemma}
\begin{proof}
Suppose $T_1=T_{11}\oplus T_{12}\oplus \cdots \oplus T_{1m}$ and
$\mathrm{\underline{dim}}T_1=(d_1,d_2,\cdots,d_r)$ where
$T_{1i}(1\leq i\leq m)$ are indecomposable regular modules with
quasi-socle $E_{i_{1}}$ and
$\underline{dim}T_{1i}=(d_{1i},d_{2i},\cdots, d_{ri})$ for $1\leq
i\leq m.$ Thus, we can see that
$(d_1,d_2,\cdots,d_r)=\sum_{i=1}^{m}(d_{1i},d_{2i},\cdots, d_{ri}).$
By Corollary \ref{DM1} and Theorem \ref{16}, we have
\begin{eqnarray*}
   && X^{d_1}_{E_1}X^{d_2}_{E_2}\cdots X^{d_r}_{E_r}\nonumber \\
   &=& \prod_{i=1}^{m}(X_{E_{i_{1}}}X_{E_{i_{1}+1}}X_{E_{i_{1}+2}}\cdots X_{E_{i_{1}+d_{1i}+\cdots+d_{ri}-1}}) \nonumber \\
   &=& \prod_{i=1}^{m}(X_{T_{1i}}+\sum_{\mathrm{\underline{dim}}L'\prec
\mathrm{\underline{dim}}T_{1i}}a_{L'}X_{L'})\nonumber \\
   &=& X_{T_{1}}+\sum_{\mathrm{\underline{dim}}L\prec
\mathrm{\underline{dim}}T_{1}}a_{L}X_{L}.
\end{eqnarray*}
where $a_{L'},a_{L}$ are integers. Similarly we have
$$X^{d_1}_{E_1}X^{d_2}_{E_2}\cdots X^{d_n}_{E_n}=
 X_{T_2}+\sum_{\mathrm{\underline{dim}}M\prec
\mathrm{\underline{dim}}T_2}b_{M}X_{M}$$ where $b_{M}$ are integers.

 Thus
$$X_{T_{1}}+\sum_{\mathrm{\underline{dim}}L'\prec
\mathrm{\underline{dim}}T_{1}}a_{L}X_{L}=X_{T_2}+\sum_{\mathrm{\underline{dim}}M\prec
\mathrm{\underline{dim}}T_2}b_{M}X_{M}.$$
 Therefore, we have
$$X_{T_1}=X_{T_2}+\sum_{\mathrm{\underline{dim}}R\prec
\mathrm{\underline{dim}}T_2}a_{R}X_{R}$$ where $a_{R}$ are integers.
\end{proof}
We explain the method used in Lemma \ref{6} by the following
example.
\begin{Example}
Consider $r=4,X_{E_2[5]}$ and $X_{E_1[4]\oplus E_2}$. We can see
that $\mathrm{\underline{dim}}(E_1[4]\oplus
E_2)=\mathrm{\underline{dim}}E_2[5]=\mathrm{\underline{dim}}(E_1\oplus
2E_2\oplus E_3\oplus E_4)$ and satisfy the conditions in Lemma
\ref{6}. Thus, for $X_{E_1[4]\oplus E_2}$, we have
\begin{eqnarray*}
X_{E_1}X^{2}_{E_2}X_{E_3}X_{E_4}&=& X_{E_1}X_{E_2}X_{E_3}X_{E_4}X_{E_2}\nonumber \\
   &=& (X_{E_1[2]}+1)X_{E_3}X_{E_4}X_{E_2}\nonumber \\
   &=& (X_{E_1[3]}+X_{E_1})X_{E_4}X_{E_2}+X_{E_3}X_{E_4}X_{E_2}\nonumber \\
   &=& (X_{E_1[4]}+X_{E_1[2]})X_{E_2}+X_{E_1}X_{E_4}X_{E_2}+X_{E_3}X_{E_4}X_{E_2}\nonumber \\
   &=& X_{E_1[4]\oplus E_2}+X_{E_1[2]\oplus E_2}+(X_{E_1[2]}+1)X_{E_4}+(X_{E_2[2]}+1)X_{E_4}\nonumber \\
   &=& X_{E_1[4]\oplus E_2}+X_{E_1[2]\oplus E_2}+X_{E_1[2]\oplus E_4}+X_{E_2[3]}+X_{E_2}+2X_{E_4}.
\end{eqnarray*}
Similarly for $X_{E_2[5]}$, we have
\begin{eqnarray*}
X_{E_1}X^{2}_{E_2}X_{E_3}X_{E_4}&=& X_{E_2}X_{E_3}X_{E_4}X_{E_1}X_{E_2}\nonumber \\
   &=& (X_{E_2[2]}+1)X_{E_4}X_{E_1}X_{E_2}\nonumber \\
   &=& (X_{E_2[3]}+X_{E_2})X_{E_1}X_{E_2}+X_{E_4}X_{E_1}X_{E_2}\nonumber \\
   &=& (X_{E_2[4]}+X_{E_2[2]})X_{E_2}+X_{E_2}X_{E_1}X_{E_2}+X_{E_4}X_{E_1}X_{E_2}\nonumber \\
   &=& X_{E_2[5]}+X_{E_2[3]}+X_{E_2[2]\oplus E_2}+(X_{E_1[2]}+1)X_{E_2}+(X_{E_1[2]}+1)X_{E_4}\nonumber \\
   &=& X_{E_2[5]}+X_{E_2[3]}+X_{E_2[2]\oplus E_2}+X_{E_1[2]\oplus E_2}+X_{E_2}+X_{E_1[2]\oplus E_4}+X_{E_4}.
\end{eqnarray*}
Hence, $X_{E_1[4]\oplus E_2}=X_{E_2[5]}+X_{E_2[2]\oplus
E_2}-X_{E_4}$, where $\mathrm{\underline{dim}}(E_2[2]\oplus
E_2)\prec \mathrm{\underline{dim}}E_2[5],
\mathrm{\underline{dim}}E_4\prec \mathrm{\underline{dim}}E_2[5].$
\end{Example}

\begin{Lemma}\label{lem}
$$X_{E_i[r]}=X_{E_{i+1}[r-2]}+2.$$
\end{Lemma}
\begin{proof}
According to Proposition \ref{E[n]}.
\end{proof}

\begin{Lemma}\label{7}
For any $M, N\in$ $\mathrm{mod} kQ$,  $X_MX_N$ is a
$\mathbb{Z}$-linear combination of the elements in $\mathcal{B}(Q).$
\end{Lemma}

\begin{proof}
By  Theorem \ref{16}, we know that $X_{M}X_{N}$ must be a
$\mathbb{Z}$-linear combination of elements in the set
$$\{X_{T\oplus
R}|\mathrm{Ext}_{k Q}^{1}(T,R)=\mathrm{Ext}_{k Q}^{1}(R,T)=0\}$$
where $R$ is $0$ or any regular exceptional module and $T$ is $0$ or
any indecomposable regular module with self-extension.

By Lemma \ref{6} and Lemma \ref{lem}, we can easily find that
$X_{M}X_{N}$ is actually a $\mathbb{Z}$-linear combination of
elements in the set $\mathcal{B}(Q)$.
\end{proof}

\begin{Prop}\label{theorem1}
Let $\Omega=\{A=(a_{ij})\in M_{r\times r}(\bbz_{\geq 0})\mid
a_{i,r}\cdots a_{i,r+i-2}\neq 0\}$ where $a_{i, r+s}=a_{i,s}$ for
$i\geq 2$ and $s\in \bbn$. Let $E(A, i)=E_{i}^{a_{i,1}}\oplus \cdots
\oplus E_{i+r-2}^{a_{i, r-1}}$ for $A\in \Omega$ and $i=1, \cdots,
r.$ Then the set $\mathcal{B'}(Q)=\{X_{E(A, i)}\mid i=1, \cdots, r,
A\in \Omega\}$ is a linearly independent set over $\bbz.$
\end{Prop}
\begin{proof}
Suppose that there exists the identity $ S:=\sum_{A\in \Omega_0,
i=1,\cdots r}n(A,i)X_{E(A, i)}=0$ where $\Omega_0$ is a finite
subset of $\Omega$ and $n(A,i)\neq 0\in \bbz$ for $i=1, \cdots, r.$
Note that
$$
X_{E(A, i)}=\prod_{j=i}^{i+r-2}(\frac{x_{j+1}+x_{j-1}}{x_j})^{a_{i,
j-i+1}}
$$
for $i=1, \cdots, r.$ Define a lexical order by set
$x_r<x_1<x_2<\cdots <x_{r-1}$ and $x_i^{a}<x_i^b$ if $a<b$. Set
$l_r(A)=max \{a_{2, r-1}, \cdots, a_{r,1}\}$ and
$l_r=max\{l_r(A)\}_{A\in \Omega_0}$. Then $l_r\neq 0.$ Note that
$a_{i, r-i+1}$ is just the exponent of $X_{E_{r}}$ in the expression
of $X_{E(A, i)}$ for $i=2, \cdots, r.$ Then the expression of $
\sum_{A\in \Omega_0, i=1,\cdots r}n(A,i)X_{E(A, i)}$ contains the
unique part of the form $\frac{L(x_1,\cdots,x_{r-1})}{x^{l_r}_r}$
which has the minimal exponent at $x_r$ and $L(x_1,\cdots,x_{r-1})$
is a Laurent polynomial associated to $x_1,\cdots,x_{r-1}$. In fact,
$\frac{L(x_1,\cdots,x_{r-1})}{x^{l_r}_r}$ is  a part of the sum
 $\sum_{i=2,\cdots r, A\in\Omega_0;
a_{i,r-i+1}=l_r}n(A, i)X_{E(A, i)}$. Note that the terms in this sum
have a common factor $X^{l_r}_{E_{r}},$ thus we have the following
identity
$$\sum_{i=2,\cdots r, A\in\Omega_0;
a_{i,r-i+1}=l_r}n(A, i)X_{E(A, i)}=(\ast)X^{l_r}_{E_{r}}$$ here we
 denote $\frac{1}{X^{l_r}_{E_{r}}}\sum_{i=2,\cdots r, A\in\Omega_0;
a_{i,r-i+1}=l_r}n(A, i)X_{E(A, i)}$ by $(\ast)$.

 Now we
set $l_{r+1}(A)=max \{a_{3, r-1}, \cdots, a_{r,2}\}$ and
$l_{r+1}=max\{l_r(A)\}_{A\in \Omega_0}$. Then $l_{r+1}\neq 0.$ In
the same way as above, we know that the expression of the term
$(\ast)$ contains the unique part of the form
$\frac{L(x_2,\cdots,x_{r-1})}{x^{l_{r+1}}_1}$ which has the minimal
exponent at $x_1=x_{r+1}$ and $L(x_2,\cdots,x_{r-1})$ is a Laurent
polynomial associated to $x_2,\cdots,x_{r-1}$. Note that
$\frac{L(x_2,\cdots,x_{r-1})}{x^{l_{r+1}}_1}X^{l_r}_{E_{r}}$ is
actually a part of the following term
$$\sum_{i=3,\cdots r,
A\in\Omega_0; a_{i,r-i+1}=l_r, a_{i, r-i+2}=l_{r+1}}n(A, i)X_{E(A,
i)}=(\ast\ast)X^{l_{r+1}}_{E_{1}}X^{l_r}_{E_{r}}$$ here  we
 denote $\frac{1}{X^{l_{r+1}}_{E_{1}}X^{l_r}_{E_{r}}}\sum_{i=3,\cdots r,
A\in\Omega_0; a_{i,r-i+1}=l_r, a_{i, r-i+2}=l_{r+1}}n(A, i)X_{E(A,
i)}$ by $(\ast\ast)$. Continue this discussion, we deduce that there
exists some $n(A, i)=0$. It is a contradiction.
\end{proof}
\begin{Thm}\label{theorem2}
The set $\mathcal{B}(Q)$ is a $\mathbb{Z}$-basis of the algebra
$\mathcal{AH}(Q).$
\end{Thm}
\begin{proof}
It is easy to prove that the elements in $\mathcal{B'}(Q)$ and
elements in $\mathcal{B}(Q)$ have a unipotent matrix transformation.
Then by Proposition \ref{theorem1} and Lemma \ref{7}, we know that
$\mathcal{B}(Q)$ is a $\mathbb{Z}$-basis of the algebra
$\mathcal{AH}(Q).$
\end{proof}
\begin{Example}\label{exam}
(1) Consider $r=1$, then we can calculate
$$X_{E_1}=2, X_{E_1[2]}=3, X_{E_1[3]}=4, \cdots, X_{E_1[n]}=n+1,\cdots$$
It is obvious that $\mathcal{B}(Q)=\{1\}$.

(2) Consider $r=2$, then we can calculate
$$X_{E_1}=\frac{2x_2}{x_1}, X_{E_1[2]}=3, \cdots, X_{E_1[2n-1]}=\frac{2nx_2}{x_1}, X_{E_1[2n]}=2n+1,\cdots$$
$$X_{E_2}=\frac{2x_1}{x_2}, X_{E_2[2]}=3, \cdots, X_{E_2[2n-1]}=\frac{2nx_1}{x_2}, X_{E_2[2n]}=2n+1,\cdots$$
It is obvious that $\mathcal{B}(Q)=\{X^{m}_{E_1},X^{n}_{E_2}| m,n\in
\mathbb{Z}_{\geq 0}\}$.
\end{Example}

\section*{Acknowledgements}
The authors are grateful
 to Professor Jie Xiao for helpful discussions.

\end{document}